\documentclass[11pt]{amsart}

\usepackage{amssymb,amsmath,amscd,graphicx,
latexsym,amsthm}
\usepackage{amssymb,latexsym,amsmath,amscd,graphicx}
\usepackage[mathscr]{eucal}
  \usepackage[all]{xy}
\def\PP{{\mathbb P}}

\def\OO{{\mathcal O}}

\def\F{\mathcal{F}}
\def\E{\mathcal{E}}

\def\cP{\mathcal{P}}
\def\Pic0{{\rm Pic}^0}

\def\*{{\underline *}}

\theoremstyle{plain}

\newtheorem{theorem}{Theorem}[section]

\newtheorem{proposition/example}[theorem]{Proposition/Example}

\newtheorem{corollary}[theorem]{Corollary}
\newtheorem{lemma}[theorem]{Lemma}

\theoremstyle{definition}

\newtheorem{remark}[theorem]{Remark}

\newtheorem{conjecture/question}[theorem]{Conjecture/Question}

\newtheorem{remark/definition}[theorem]{Remark/Definition}
\newtheorem{notation/assumptions}[theorem]{Assumptions/Notation}

\numberwithin{equation}{section}

 \theoremstyle{remark}

\title{ 2-torsion points on theta divisors}
 \author[ G. Pareschi,  R. Salvati Manni]{  Giuseppe Pareschi and Riccardo Salvati Manni}

\address{G. Pareschi\\Dipartimento di Matematica,
              Universit\`a di Tor Vergata\\Italy}
\email{pareschi@mat.uniroma2.it}

\address{R. Salvati Manni \\Dipartimento di Matematica ``Guido Castelnuovo'', Universit\`a di Roma ``La Sapienza"\\Italy}
\email{salvati@mat.uniroma1.it}
 \thanks{GP was partially supported by the MIUR Excellence Department Project awarded to the Department of Mathematics, University of Rome Tor Vergata, CUP E83C18000100006".}


\begin{document}
\maketitle
\begin{abstract}
In this  note we prove a sharp  bound  for the number of 2-torsion points on a  theta divisor and show that this is achieved only in the case of products of elliptic curves. This settles in the affirmative a conjecture of Marcucci and Pirola.

\end{abstract}

\section{Introduction}

Let $A$ be a $g$-dimensional principally polarized abelian variety. Let $\Theta$ be a  divisor on $A$ representing the principal polarization. We set  
$$\Theta(n):=\#A[n]\cap \Theta,$$
where $A[n]$ is the group of $n$-torsion points on $A$. The geometry of $\Theta(2)$ is interesting. In fact, once  a symmetric divisor $\Theta$ is chosen,  $A[2]$  splits in  the even and odd part. The odd points   lie in $\Theta$.  On the other hand, the even  ones which lie in $\Theta$  correspond  to the classical thetanulls. In the complex case,  their  presence provides a characterization of  decomposable  principally polarized abelian variety   (\cite{S}, \cite{DGPS}) or 
a hyperelliptic jacobian (\cite{Mu}).  

In \cite{PM}, Marcucci and Pirola  gave  a bound   for  $\Theta(2)$  (over $\mathbb C$). This  bound  has been recently improved by Auffarth, Pirola and Salvati Manni in  \cite{APS} where also a bound for $\Theta(n)$  is given.  However, these bounds were  not  optimal. In   \cite{PM} it  has been    conjectured that the maximal $\Theta(2)$ is computed exactly by products of elliptic curves, and a similar conjecture for $\Theta(n)$ has been formulated in \cite{APS}. The purpose of this note is to prove the  conjecture for $\Theta(2)$. Our methods are algebraic and work in  characteristic $\ne 2$. 
 Thus  we will  characterize  principally polarized abelian variety  with the  maximum number of thetanulls.

\begin{theorem} \label{main} Let $A$ be an principally polarized abelian variety over an algebraically closed field of characteristic $\ne 2$.  For all divisors $\Theta$ representing the principal polarization
$$\Theta(2)\leq 4^{g}- 3^g.$$
Moreover equality  holds  if and  only if  $A$  is a product of  elliptic  curves and $\Theta$ is a symmetric theta divisor of $A$.
\end{theorem}

\subsection{Acknowledgement} We are grateful  to Robert Auffarth   and Pietro Pirola for  many discussions, suggestions and for their  continuous  encouragement. The proof of  Kempf's Theorem \ref{rank} outlined in Remark \ref{pareschi-popa} 
originates from work done years ago with Mihnea Popa.

\section{Preliminaries}
Let   $\Theta$  a \emph{symmetric} theta divisor representing the principal polarization.  All divisors representing the principal polarization are translates of $\Theta$. Let us denote $L=\OO_A(\Theta)$. We denote $t_x:A\rightarrow A$ the translation by $x$. 
For $x,y\in A$ we consider the multiplication map of global sections
\begin{equation} M(x,y):H^0(A,t_x^*L^2)\otimes H^0(A,t_y^*L^2)\rightarrow H^0(A, t_x^*L^2\otimes t_y^*L^2)
\end{equation}
\begin{theorem}\label{rank} \emph{(Kempf) } The rank of the map $M(x,y)$ is equal to the number of $\eta\in A[2]$ such that 
\[y-x+\eta\not\in \Theta\]
Hence  $(\Theta+y)(2)$  is equal to the rank of  $M(0,y)$ for all $y\in A$. 
\end{theorem} 

This has as a corollary the following 
\begin{theorem}\label{kempf1}
The map $M(x,y)$ is not surjective if and only if
\[y-x\in\bigcup_{\eta\in A[2]}\Theta+\eta\]
In particular, for any fixed $x\in A$,  the multiplication map $M(x,y)$ is surjective for general $y\in A$
\end{theorem}

In turn Theorem \ref{kempf1} has the following consequence, originally proved by Koizumi (\cite{koizumi})
\begin{theorem}\label{koizumi}\emph{(Koizumi)} For $h,k\ge 2$ and $h+k\ge 5$ the multiplication map of global sections
\[H^0(A,t_x^*L^h)\otimes H^0(A,t_y^*L^k)\rightarrow H^0(A, t_x^*L^h\otimes t_y^*L^k)\]
is surjective for all $x,y\in A$.
\end{theorem}

Theorem \ref{koizumi} is well known and admits different proofs (see e.g. \cite[Thm 6.8(c)]{K1}, \cite[Example 3.7]{pareschi}, \cite[Thm D]{JP}). 

The last assertion of Theorem \ref{kempf1} was proved independently by Kempf (\cite[Thm 1]{Kempf} and Sekiguchi (\cite[Prop. 1.5]{Se}). 
  A  similar statement, in the more general context of any ample line bundle, was proved independently by Ohbuchi (\cite{O}). A proof not using theta-groups, but rather the Fourier-Mukai transform associated to the Poincar\'e line bundle, was subsequently given by Pareschi and Popa in \cite[Thm 5.8]{PP}.
  
  The more precise formula of Theorem \ref{rank} and, consequently, the formula in Theorem \ref{kempf1} are again due to Kempf (\cite[Thm 3]{Kempf}).\footnote{A word of warning about \cite[Thm 3]{Kempf}. 
There is a small mistake in the statement  there: the displayed formula  should read: $2z+\eta\not\in\Theta+2x$ (this is equivalent to the statement of  Theorem \ref{rank} here). 
In fact there is an inaccuracy in the proof: the beginning of the first line of page 773 should be: $2s\not\in\theta+\eta$. }

\begin{remark}\label{pareschi-popa} Although Theorem \ref{rank} is not  mentioned in the paper \cite{PP}, it could have been re-proved there within the same Fourier-Mukai methods, by means of an additional argument.  For the benefit of  readers  more familiar with such methods,  we do it here. To this purpose  in the first place we notice that
 in Theorem \ref{rank} it is sufficient to assume that $x=0$. 
The approach of  \cite{pareschi}, and subsequently \cite[Thm 5.8]{PP} identifies the multiplication map $M(0,y)$ with the evaluation at the point $y\in A$ of the global sections  of the vector bundle (called "skew Pontryagin product", see \cite{PP} Terminology 5.1)) 
$\E:=L^2\hat * L^2$. Therefore we denote such evaluation map in the same way:
\[M(0,y): H^0(\E)\rightarrow \E(y)\]  
We recall that the "WIT criterion" given by \cite[Thm 4.1]{PP} yields that a vector bundle $\F$ is \emph{weakly continuously globally generated} as soon as the dual vector bundle $\F^\vee$ satisfies WIT(g) plus some other technical conditions. In the proof of \cite[Thm 5.8]{PP}  it is shown that the WIT(g) condition (as well as the additional technical conditions) is satisfied by the vector bundle $\bigl(\E\otimes L^{-1}\bigr)^\vee$ (in the more general context where $L$ can be any ample line bundle on $A$). 
Therefore the vector bundle $\E\otimes L^{-1}$ is \emph{weakly continuously globally generated}. Moreover  in the same proof it is also  explicitly computed  the  Fourier-Mukai transform  of the vector bundle $\bigl(\E\otimes L^{-1}\bigr)^\vee$ (see (1) in  the proof of \emph{loc.cit.}), and it is shown that such calculation implies, in the case at hand,  that to be weakly continuously globally generated concretely means that the sum of evaluation maps
\[\bigoplus_{\alpha\in \Pic0 A[2]}H^0(\E\otimes L^{-1} \otimes\alpha)\otimes \alpha\rightarrow \E\otimes L^{-1}. \]
is surjective. Again all this works more generally for any ample line bundle on an abelian variety as well. If $L$ is a \emph{principal} polarization then  $h^0(\E\otimes L^{-1} \otimes\alpha)=1$ for all $\alpha\in \Pic0 A[2]$ (this is again contained in (1) of the proof of \emph{loc.cit.}). Therefore for all $y\in A$ we have the decomposition as direct sum of $1$-dimensional subspaces
\[\bigoplus_{\alpha\in \Pic0 A[2]}H^0(\E\otimes L^{-1} \otimes\alpha)\otimes \alpha (y)\cong  (\E\otimes L^{-1})(y)\]
The commutative diagram
\[\xymatrix{\bigoplus_{\alpha\in \Pic0 A[2]}H^0(\E\otimes L^{-1} \otimes\alpha)\otimes H^0(L\otimes \alpha)\ar[r]\ar[d]& H^0(\E)\ar[d]^{M(0,y)}\\
\bigoplus_{\alpha\in \Pic0 A[2]}H^0(\E\otimes L^{-1} \otimes\alpha)\otimes (L\otimes \alpha)(y)\ar[r] &\E(y)}\]
shows that the rank of the map $M(0,y)$ is precisely the number of line bundles  $\alpha\in\Pic0 A[2]$ such that the evaluation of $H^0(L\otimes \alpha)$ at $y$ is non-zero.  One gets the statement of Theorem \ref{rank} via the isomorphism $\varphi_{L}:A\rightarrow 
\Pic0 A$ sending $x$ to $(t_x^*L)\otimes L^{-1}$.
\end{remark}

\section{Proof of Theorem \ref{main}}

\subsection{ } We keep the notation introduced in the previous section. Specifically, let $L=\OO_A(\Theta)$, where $\Theta$ is a symmetric theta divisor. For a point  $x\in A$ we denote $V_x$ the image of the multiplication map $M(0,x)$. For all $x,y\in A$ let us consider the following commutative diagram of multiplication maps of global sections
 \begin{equation}\label{diagram}\xymatrix{H^0(L^2)\otimes H^0(t_x^*L^2)\otimes H^0(t_y^*L^2)\ar[rr]^{\>\>\>\>\>\>\mathrm{id}\otimes M(x,y) } \ar[dd]^{M(0,x)\otimes \,\mathrm{id}}&&H^0(L^2)\otimes H^0(t_x^*L^2\otimes t_y^*L^2)\ar[dd]\\ \\
  V_x\otimes H^0(t_y^*L^2)\ar[rr]^{ N(x,y)}&&H^0(L^2\otimes t_x^*L^2\otimes t_y^*L^2)}
  \end{equation}
  By Theorem \ref{koizumi} the right vertical arrow is surjective for all $x,y\in A$,   since,   by the theorem of the square, $t_x^*L^2\otimes t_y^*L^2= (t_z^*L)^4$  for a suitable $z$. Let us fix a point $x\in A$. By Theorem \ref{kempf1} the map $\mathrm{id}\otimes M(x,y)$ is surjective for general $y\in A$. Therefore, for such $y$'s, the map $N(x,y)$ is surjective, hence
  $\dim V_x\cdot 2^g\ge 6^g$. In conclusion for all $x\in A$ the rank of the map $M(0,x)$ is $\>\ge 3^g$. Therefore, by Theorem \ref{rank} 
  \[(\Theta+x)(2)\le 4^g-3^g.\]
  for all $x\in A$. Since every divisor  representing the principal polarization is a translate of a symmetric one,  this proves the first assertion of Theorem \ref{main}. 
  
  \subsection{ } To prove the last assertion we need to construct a commutative diagram of locally free sheaves on $A$ inducing diagram (\ref{diagram}) at the fiber level. To this purpose we consider a Poincar\'e line bundle $\cP$ on $A\times \Pic0A$. Given  a coherent sheaf $\F$ such that $h^i(\F\otimes P_\alpha)=0$ for all $i>0$ and for all line bundles $\alpha\in\Pic0 A$ we denote $\Phi(\F)$ the coherent sheaf on $\Pic0 A$ defined as follows:
  \[\Phi(\F)=q_*(p^*(\F)\otimes \cP)\] where $p$ and $q$ are respectively the first and second projection of $A\times\Pic0A$.\footnote{Thanks to the hypothesis on $\F$ this is in fact the Fourier-Mukai transform of $\F$.}
  By base change the hypothesis on $\F$ yields that $\Phi(\F)$ is a locally free sheaf whose fibre at the point $\alpha\in\Pic0 A$ is canonically identified to the vector space $H^0(A,\F\otimes \alpha)$.
  
   For $x\in A$ we consider the commutative diagram of locally free sheaves on $\Pic0 A$
 \begin{equation}\label{diagram1}\xymatrix{H^0(L^2)\otimes H^0(t_x^*L^2)\otimes \Phi(L^2)\ar[rr]^{\>\>\>\>\>\>\mathrm{id}\otimes \widetilde{\mathcal M_x} } \ar[dd]^{M(0,x)\otimes \,\mathrm{id}}&&H^0(L^2)\otimes \Phi(t_x^*L^2\otimes L^2)\ar[dd]\\ \\
  V_x\otimes \Phi(L^2)\ar[rr]^{ \widetilde{\mathcal N_x}}&&\Phi(L^2\otimes t_x^*L^2\otimes L^2)}
  \end{equation}
  where the maps $\widetilde{M_x}$, $\widetilde{N_x}$ and the left vertical arrow are fiberwise multiplication maps of global sections. The definition of the maps appearing in diagram (\ref{diagram1}) is left to the reader.
  
  For an ample line bundle $M$ on $A$ let us denote $\varphi_{M}:A\rightarrow \Pic0 A$ the isogeny associated to the polarization represented by $M$: $\varphi_{M}(y)=(t_y^*M)\otimes M^{-1}$. As $L$ represents a principal polarization,   $\varphi_L:A\rightarrow \Pic0 A$ is an isomorphism. Via such isomorphism, for all positive integers $k$ the isogeny $\varphi_{L^k}$ is identified to the multiplication by $k$ homomorphism $k_A:A\rightarrow A$, defined by $y\mapsto ky$. Applying $\varphi_{L^2}^*$ to diagram (\ref{diagram1}) we get
  \begin{equation}\label{diagram2}\xymatrix{H^0(L^2)\otimes H^0(t_x^*L^2)\otimes \varphi_{L^2}^*\Phi(L^2)\ar[rr]^{\>\>\>\>\>\>\mathrm{id}\otimes \mathcal M_x } \ar[dd]^{M(0,x)\otimes \,\mathrm{id}}&&H^0(L^2)\otimes \varphi_{L^2}^*\Phi(t_x^*L^2\otimes L^2)\ar[dd]\\ \\
  V_x\otimes \varphi_{L^2}^*\Phi(L^2)\ar[rr]^{ \mathcal N_x}&&\varphi_{L^2}^*\Phi(L^2\otimes t_x^*L^2\otimes L^2)}
  \end{equation} 
  By construction diagram (\ref{diagram2}) induces diagram (\ref{diagram})  at the fibre level.

 \subsection{ } Next, we make some calculations to compute the (determinant of) the locally free sheaves appearing in  diagram (\ref{diagram2}).  It is well known (\cite[Prop. 3.11]{mukai}) that, for an ample line bundle $M$ on $A$, 
  \[\varphi_{M}^*\Phi(M)\cong H^0(M)\otimes M^{-1}.\]
  Therefore
\[\varphi_{L^2}^*\Phi(L^2)\cong H^0(L^2)\otimes L^{-2}\]
Similarly, since the isogeny $\varphi_{L^k}$ is identified to the isogeny $k_A$, we have that
\[{2_A}^*\Bigl( \varphi_{L^2}^*\Phi( t_x^*L^2\otimes L^2)\Bigr)\cong H^0( t_x^*L^2\otimes L^2))\otimes  t_{-x}^*L^{-2}\otimes L^{-2}\]
 \[{3_A}^*\Bigl( \varphi_{L^2}^*\Phi(L^2\otimes t_x^*L^2\otimes L^2)\Bigr)\cong H^0(L^4\otimes t_x^*L^2))\otimes L^{-4}\otimes t_{-x}^*L^{-2}\]
Since, given a line bundle $M$ on $A$, the line bundle $k_A^*M$ is algebraically equivalent to $M^{k^2}$, after short calculations we get
\begin{equation}\label{det1}\det ( \varphi_{L^2}^*\Phi(  L^2)\sim L^{-2^{g+1}}\end{equation}
\begin{equation}\label{det2}\det ( \varphi_{L^2}^*\Phi( t_x^*L^2\otimes L^2)\sim L^{-4^g}\end{equation}
\begin{equation}\label{det3}\det(\varphi_{L^2}^*\Phi(L^2\otimes t_x^*L^2\otimes L^2))\sim L^{-4\cdot 6^{g-1}}\end{equation}
where $\sim$ means algebraic equivalence.

 \subsection { } By Theorem \ref{kempf1} the map $\mathrm{id}\otimes\mathcal M_x$  drops rank at a divisor $2^gE_x$ where $E_x$ is a divisor such that
  \[supp(E_x)=supp\Bigl(\sum_{\eta\in A[2]}\Theta+x+\eta\Bigr)\]
  By (\ref{det1}) and (\ref{det2})
  $E_x\sim 4^g\Theta$ hence 
  \[E_x=\sum_{\eta\in A[2]}\Theta+\eta+x\]

  \subsection{ } After this preparation now we are ready for the proof of the second part of Theorem \ref{main}. To begin with we claim that if$g\ge 2$ and  $\Theta$ is irreducible then the bound of Theorem \ref{main} cannot be achieved, i.e. $\dim V_x>3^g$ for all $x\in A$.  Indeed, if $\dim V_x=3^g$ for a point $x\in A$ then
   the two locally free sheaves of the bottom arrow of diagram (\ref{diagram2}) have the same rank, so that the map drops rank on a divisor $D_x$. From (\ref{det1}) and (\ref{det3}) it follows that
  \begin{equation}\label{coefficient}D_x\sim 8(6^{g-1}\Theta)\end{equation}
  By Theorem \ref{koizumi} the support of $D_x$ is contained in the support of $E_x$. However, nothing changes if we replace
   $L$ with $t_\eta^*L$, with
  $\eta\in A[2]$. Therefore both the divisors $D_x$ and $E_x$ are invariant with respect to the action of $A[2]$. Hence
  \[supp(E_x)=supp(D_x)=\sum_{\eta\in A[2]}\Theta+\eta+x\]
  In conclusion, if $\Theta$ is irreducible
  \[supp(D_x)=\sum_{\eta\in A[2]}\Theta+\eta+x\]
  Hence $D_x$ must be, up to algebraic equivalence, an integral multiple of $4^g\Theta$. But this, for $g\ge 3$, is clearly in contrast with (\ref{coefficient}). 
  Concerning the case $g=2$,  one can prove directly that  $\Theta$ is irreducible then for all $x\in A$  the rank of the map $ M(0,x)$ is $>9$. If $x\in A[2]$ this is well known, see \cite{Fr} (recall that
   $M(0,0)=M(0,x)$ for all $x\in A[2]$) .  The same thing should be known also in the case $x\not\in A[2]$. However, for lack of reference, we include a proof (of a more precise statement) in Lemma \ref{g=2} below. 
    Therefore, as claimed, assuming $g\ge 2$, if for some $x\in A$  the rank of the map $M(0,x)$ is equal to $3^g$ then $\Theta$ must be reducible.  Finally, for sake of completeness, we recall that for an elliptic curve $A$  the rank of the map $M(0,x)$ is equal to
    $3$ if and only if $x\in A[2]$.

  \subsection{ }  By the previous step and the the decomposition theorem it follows that if, for some $x\in A$, $(\Theta+x)(2)=4^g-3^g$ then
  the principally polarized abelian variety $A$ is the product of two lower dimensional principally polarized abelian varieties. At this point,   arguing by induction, we have that  $A$ is a product of elliptic curves  and   $$E_x=4^{g-1}H,\quad   D_x=6^{g-1}2H $$
 where 
  \[H=\underset{ i=1,\dots , g}{\sum_{x_i\in E_i[2]}}E_1\times\cdots\times\{x_i\}\times\cdots\times E_g.\] 
  Theorem \ref{main} follows.\endproof
  
  \ \\
  
  \begin{lemma}\label{g=2} Let $(A,\Theta)$ be a principally polarized abelian variety of dimension $2$ with $\Theta$ irreducible. Then, keeping the above notation, $rk M(0,x)\ge 11$ if $x\not\in A[2]$. 
  \end{lemma}
  \begin{proof} Since $\Theta$ is irreducible, it is a smooth curve of genus $2$, say $C$.  Keeping the above notation we have the following exact sequences
\[ 0\rightarrow H^0(L)\rightarrow H^0(L^2)\rightarrow H^0(L^2_{|C})\rightarrow 0\]
  \[  0\rightarrow H^0(L^{-1}\otimes t_x^*L^2)\rightarrow H^0(t_x^*L^2)\rightarrow H^0((t_x^*L^2)_{|C})\rightarrow 0\] 
  inducing the filtration on the tensor product of the middle vector spaces
  \begin{equation}\label{filt1}H^0(L)\otimes H^0(L^{-1}\otimes t_x^*L^2)\subset \>U\subset \> H^0(L^2)\otimes H^0(t_x^*L^2)
  \end{equation}
  where
  \[\frac{U}{H^0(L)\otimes H^0(L^{-1}\otimes t_x^*L^2)}\cong \bigl(H^0(L)\otimes H^0((t_x^*L^2)_{|C})\bigr)\oplus \bigl(H^0(L^{-1}\otimes t_x^*L^2)\otimes H^0(L^2_{|C})\bigr)\]
  and
  \[ \frac{H^0(L^2)\otimes H^0(t_x^*L^2)}{U}\cong H^0(L^2_{|C})\otimes H^0((t_x^*L^2)_{|C})\]
  We have also the filtration
   \begin{equation}\label{filt2} H^0( t_x^*L^2)\subset H^0(L\otimes t_x^*L^2)\subset H^0(L^2\otimes t_x^*L^2)
   \end{equation}
   where obviously
   \[\frac{H^0(L\otimes t_x^*L^2)}{H^0( t_x^*L^2)}\cong H^0(L\otimes t_x^*L^2_{|C})\quad\hbox{and}\quad\frac{H^0(L^2\otimes t_x^*L^2)}{H^0(L\otimes t_x^*L^2)}\cong H^0(L^2\otimes t_x^*L^2_{|C})\]
   The multiplication map of global sections $M(0,x)$ is in fact a map of filtered vector spaces from (\ref{filt1}) to (\ref{filt2}). Since $x\not\in A[2]$, the line bundles $L^2_{|C}$ and $(t_x^*L^2)_{|C}$ do not coincide. This being  the case the multiplication map of global sections $H^0(L^2_{|C})\otimes H^0((t_x^*L^2)_{|C})\rightarrow H^0((L^2\otimes t_x^*L^2)_{|C})$ is surjective, hence of rank 7. This could be proved directly, however it is a (very) particular case of a general result  of Eisenbud, Koh and Stillman about multiplication maps of global sections on curves, and even syzygies, see \cite[Thm 2]{eks}. 
   On the other hand, we claim that the induced map $U\rightarrow H^0(L\otimes t_x^*L^2)$ has rank $\ge 4$.  Therefore the rank of the map $M(0,x)$ is $\ge 11$, proving the Lemma.  
   To prove the claim, note that the map $H^0(L)\otimes H^0(L^{-1}\otimes t_x^*L^2)\rightarrow
   H^0( t_x^*L^2)$ has rank $1$ and the map 
   \[\bigl(H^0(L)\otimes H^0((t_x^*L^2)_{|C})\bigr)\oplus \bigl(H^0(L^{-1}\otimes t_x^*L^2)\otimes H^0(L^2_{|C})\bigr)\rightarrow H^0((L\otimes t_x^*L^2)_{|C})\]
   has rank $3$. Indeed, while  the restriction to the first  summand is zero, the restriction to the second summand is injective. Indeed, since $x\not\in A[2]$, $L^{-1}\otimes t_x^*L^2$ is not isomorphic to $L$, the restriction to $C$ of  the $1$-dimensional space of global sections $H^0(L^{-1}\otimes t_x^*L^2)$ is non-zero. 
 \end{proof}
  
  \section{Variants, questions and remarks}
\subsection{}  We start with a remark concerning Theorem \ref{main}. We  observe that only  when $A$ 
is a product of elliptic curves and $\Theta$ is a symmetric theta
divisor of $A$ the map induced  by $|2\Theta|$ is of degree $2^g$ and the image is smooth. In fact it is, up to a projectivity, the  image of the   Veronese  map of  $g$ copies of $\PP^1$ into $\PP^{ 2^g-1}$.  In this case the image is defined  by the intersection of $2^{g-1}(2^g+1)- 3^g$ quadrics  that  are a basis of   the  kernel of  the  map 
\begin{equation}  \mathrm{Sym}M(0,0):\mathrm{Sym}^2 H^0(A,L^2)\rightarrow H^0(A,  L^4)
\end{equation}

\subsection{} 
 Theorem \ref{main} yields a (non-optimal) bound  for the numbers $\Theta(n)$ for $n=2m$ even,  when the characteristic of the  field does not divide $n$. In fact, since  $A[2m]/A[2]\equiv (\mathbb Z/m\mathbb Z)^{2g}$, Theorem \ref{main} implies
  \begin{corollary}\label{ntor}
$$\Theta(2m)\leq m^{2g}(4^g-3^g).$$
\end{corollary}

For $n$-torsion
points the bound should be $n^{2g} -(n^{2}-1)^g$, with equality if and only if $(A,L)$ is the polarized product of elliptic curves, cf.\cite{APS}.  A certain evidence  for this can be deduced  
from Theorem \ref{kempf1}.  In fact we  know that
  $M(0,y)$ is not surjective if and only if
\[y \in\bigcup_{\eta\in A[2]}\Theta+\eta\]
In particular  the multiplication map $M(0,y)$ is surjective for general $y\in A$.

 Now assume that $y $ is a $n=2m$ torsion point,   
 thus     we  have  

$$\Theta(n)=\bigoplus_{y\in A[n]/A[2]} \dim\ker ( M(0,y)).$$

 When $n>>0$,  because of the density of torsion points,   in mostly of cases we have that the  dimension is 0.  Hence we  can improve  the result of  Corollary \ref{ntor} with the following estimate 
 $$\Theta(n)\leq  ( 4^g- 3^g)K_n  m^{2g-2}   $$
 for a suitable constant $K_n$.

\subsection{} In the proof of Theorem \ref{main} an important role is played by the multiplication maps of global sections  
$M(0,y): H^0(A, L^2)\otimes H^0(A, t_y^*L^2)\rightarrow H^0(A, L^2\otimes t_y^*L^2)$, where $L$ represents the principal polarization. 
A crucial aspect of such maps is that dimensions of the source and of the target are equal. This doesn't happen for $n$ higher than $2$. However  in the paper \cite{JP}  natural analogues of the above maps   were introduced. 
These are the  "fractional"  multiplication maps of global sections 
\[H^0(A, L^n)\otimes H^0(A, t_y^*L^{n(n-1)})\rightarrow H^0(A, (n\!-\!1)_A^*(L^n)\otimes t_y^*L^{n(n-1)})\]
obtained by factoring with the natural inclusion of the first factor into $H^0(A, (n\!-\!1)_A^*L^n)$, 
where $(n\!-\!1)_A:A\rightarrow A$ is the isogeny $x\mapsto (n-1)x$. 
 We refer to the discussion after the proof of Corollary 8.2 of \cite{JP}  for some explanation of the attribute "fractional" as well as why such maps play the same role of the maps $M(0,y)$ for higher $n$. At present we are not able to treat them as 
  we do for "integral" multiplication maps of global sections, but we hope to come back to this in the future.
 
\subsection{ } Next, we notice that, in the complex case,   in the paper of Marcucci and Pirola, Theorem \ref{main} yields a lower bound on the number of non-effective square roots of line bundles of degree $\le g-1$ on curves of genus $g$, as well as related
bounds 
(see \cite[Prop. 2.1]{PM}).

\subsection{ }  It  would be  nice  to  get an estimate of 
$\Theta(2)$  when   $(A,\Theta)$ is an irreducible  principally polarized abelian variety.
At the best of our knowledge, in the complex case,   we have that if $(A, \Theta)$ is  the  jacobian of a hyperelliptic  curve and $\Theta$ is symmetric thus
 $$\Theta(2)=  4^g- {{2g+1}\choose g}.$$
 
 This is  a well known   result, cf.\cite{Mu}.  It seems natural to guess that  this is the maximal number (see also \cite{Po}).  In  other known  examples, it  appears   fundamental the existence of an involution, 
  cf. \cite{bea} and \cite{vg}.  
 Nevertheless in genus 4 there is an irreducible principally polarized abelian variety  that is  not a jacobian with 
 $\Theta(2)= 130$, the same as the  hyperelliptic  case (see \cite{var} or \cite{deb}). In this case the  period matrix is  related to the  Gaussian lattice $E_8$.   In  higher dimension   this class of examples have  a number of vanishing thetanulls which is smaller
 compared to the hyperelliptic Jacobians.
  However, as pointed out in  \cite{bea2}, in  a particular  subcase,  these vanishing thetanulls  have the particular property of being syzygetic (in the classical terminology).  
  It would also be interesting to get a  bound for this special case.

\end{document}